\documentclass{amsart}

\newtheorem{theorem}{Theorem}[section]
\newtheorem{lemma}[theorem]{Lemma}

\newtheorem{fact}[theorem]{Fact}

\theoremstyle{definition}
\newtheorem{definition}[theorem]{Definition}

\theoremstyle{remark}

\numberwithin{equation}{section}

\newcommand{\bfa}{{\bf a}}
\newcommand{\bfb}{{\bf b}}
\newcommand{\bfc}{{\bf c}}

\newcommand{\bfx}{{\bf x}}

\newcommand{\bfo}{{\bf o}}

\newcommand{\bfs}{{\bf s}}

\newcommand{\bfC}{{\mathbb C}}
\newcommand{\bfP}{{\mathbb P}}
\newcommand{\bfR}{{\mathbb R}}
\newcommand{\bfZ}{{\mathbb Z}}

\newcommand{\bfQ}{{\mathbb Q}}

\newcommand{\barj}{{\overline j}}

\newcommand{\barz}{{\overline z}}

\newcommand{\tildew}{{\widetilde w}}

\newcommand{\mapright}[1]{\smash{\mathop{   \hbox to 0.7cm{\rightarrowfill}}
  \limits^{#1}}}

\begin{document}

\title{Asymptotic Chow polystability in K\"ahler geometry}

\author{Akito Futaki}
\address{Department of Mathematics, Tokyo Institute of Technology, 2-12-1,
O-okayama, Meguro, Tokyo 152-8551, Japan}
\email{futaki@math.titech.ac.jp}


\date{May 24, 2011 }



\begin{abstract}
 It is conjectured that the existence of constant scalar curvature K\"ahler metrics will
 be equivalent to K-stability, or K-polystability depending on terminology (Yau-Tian-Donaldson conjecture).
 There is another GIT stability condition, called the asymptotic Chow polystability.
 This condition implies the existence of balanced metrics for polarized manifolds $(M, L^k)$ for all large $k$.
 It is expected that the balanced metrics converge to a constant scalar curvature metric as $k$ tends to
 infinity under further
suitable stability conditions.
In this survey article I will report on recent results saying that the asymptotic Chow polystability does not
hold for certain constant scalar curvature K\"ahler manifolds.  We also compare a paper of Ono with that of  Della Vedova and Zuddas.
\end{abstract}

\maketitle

\section{Introduction}

A pair $(M,L)$ of a compact complex manifold $M$ and
a positive line bundle $L$ over $M$ is called a polarized manifold.
Here a positive line bundle means a holomorphic line bundle $L$ such that
its first Chern class $c_1(L)$ is represented, as a de Rham class, by a positive closed $(1,1)$-form.
Therefore we can find a closed 2-form $\omega$ of the form
\begin{equation}
\omega = \frac i{2\pi} \sum_{i,j=1}^m\,g_{i\barj}\, dz^i \wedge d\bar z^j
\end{equation}
with $g = (g_{i\barj})$ being pointwise a positive definite Hermitian matrix,
and $z^1, \cdots, z^m$ local holomorphic coordinates.
Then $g$ defines a Hermitian metric of $M$, and $\omega$ is regarded as its fundamental 2-form.
Since $\omega$ is closed, $g$ becomes a K\"ahler metric.

Hence, for a polarized manifold $(M,L)$,
$c_1(L)$ is regarded as a K\"ahler class.
We seek a constant scalar curvature K\"ahler (cscK) metric with its K\"ahler form
in $c_1(L)$.

There are known obstructions related to holomorphic vector fields.
One is reductiveness of the Lie algebra $\mathfrak h(M)$ of all holomorphic vector fields on $M$ (\cite{Lic}, \cite{matsushima57}), and
the other is certain Lie algebra character $f : \mathfrak h(M) \to \bfC$ (\cite{futaki83.1}, \cite{calabi85}).
Besides them, there are obstructions
related to GIT stability. A well-known conjecture due to Yau, Tian, and Donaldson says the existence of constant scalar curvature metrics
in $c_1(L)$ will be equivalent to K-(poly)stability (\cite{donaldson02}).
K-stability is defined using the so-called {\it DF-invariant} as a numerical invariant for the Hilbert-Mumford criterion, see Definition \ref{DF}.
At the moment of this writing, it has been proved that the existence implies
K-stability (\cite{chentian04}, \cite{donaldson05}, \cite{stoppa0803}, \cite{mabuchi0910}), but it is still open whether K-stability implies the existence. Therefore at least K-stability is an obstruction to the existence. But there is
another stability condition which is an obstruction to the existence of cscK metrics when the automorphism group $\mathrm{Aut}(M,L)$ is discrete.
Here $\mathrm{Aut}(M,L)$ is the subgroup of the automorphism group $\mathrm{Aut}(L)$ of $L$ consisting of all automorphisms of $L$
commuting with the $\bfC^{\ast}$-action on the fibers. Notice that such automorphisms descend to automorphisms of $M$. Therefore $\mathrm{Aut}(M,L)$
is naturally identified with a subgroup of the automorphism group $\mathrm{Aut}(M)$ of $M$. From now on we regard $\mathrm{Aut}(M,L)$ as a subgroup
of $\mathrm{Aut}(M)$ in this way, and also the Lie algebra $\frak h_0$ of $\mathrm{Aut}(M,L)$ as a Lie subalgebra of the Lie algebra $\mathfrak h(M)$
of $\mathrm{Aut}(M)$.
The following result due to Donaldson shows in fact asymptotic Chow stability is an obstruction to the existence of cscK metrics.

\begin{theorem}[Donaldson \cite{donaldson01}]\label{Donaldson}
Let $(M,L)$ be a polarized manifold with $\mathrm{Aut}(M,L)$ discrete.
Suppose there exists a cscK metric in $c_1(L)$. Then $(M,L)$ is asymptotically Chow stable.
\end{theorem}

Note that if $(M,L^k)$ is Chow stable then there exists a ``balanced metric'' for $L^k$. Donaldson further proved in the same paper \cite{donaldson01} that
as $k \to \infty$, the balanced metrics converge to the cscK metric (assuming the existence of a cscK metric). Because of this result, we may have an
expectation of a possibility to use the convergence of the balanced metrics as a one step in the proof of the implication of stability implying existence.

But the claim of this talk is that Donaldson's theorem does not hold if $\mathrm{Aut}(M,L)$ is not discrete. In fact we explain the following result.

\begin{theorem}[Ono-Sano-Yotsutani \cite{OSY09}]\label{OSY}
There is a toric Fano 7-manifold (suggested by Nill and Paffenholtz in \cite{NillPaffen}) which is K\"ahler-Einstein but not
asymptotically Chow-semistable (polystable).
\end{theorem}
This result relies on our earlier works \cite{futaki04-1} and \cite{FOS08}. The following result of Della Vedova and Zuddas, which is also related to our work \cite{futaki04-1},
claims that there are two dimensional examples.
\begin{theorem}[Della Vedova-Zuddas  \cite{DVZ10}]\label{DVZ}
There are constant scalar curvature K\"ahler surfaces
which admit an asymptotically Chow unstable polarization.
\end{theorem}
The following result of Odaka uses a formula of DF-invariant for blow-ups along the flag ideals due to Wang \cite{Xiaowei08} and Odaka \cite{Odaka09}.
\begin{theorem}[Odaka \cite{Odaka10}]\label{Odaka}
There are examples of K-stable polarized orbifolds which are asymptotically Chow unstable.
In fact, these examples are K\"ahler-Einstein orbifolds
with finite automorphisms. Hence Donaldson's theorem does not hold for orbifolds.
\end{theorem}

Note that there is an argument without using balanced metrics to show that cscK metrics minimize the K-energy when the automorphism group
is not discrete, see Li \cite{LiChi10}.

\section{ What is (asymptotic) Chow stability ?}

Let
$V_k := H^0(M,\mathcal O(L^k))^*$ be the vector space of all holomorphic sections of $L^k$,
$M_k \subset \bfP(V_k)$ the image of Kodaira embedding by $L^k$, and
$d_k$ the degree of $M_k$ in $\bfP(V_k)$.

Denote by  $m$ the dimension of $M$: $m = \dim_{\bfC} M$.
An element of $\bfP(V_k^*) \times \cdots \times \bfP(V_k^*)$ ($m+1$ times)
defines  $m+1$ hyperplanes $H_1,\,\cdots\, ,H_{m+1}$ in $\bfP(V_k)$. Then the set
$$\{(H_1, \cdots , H_{m+1}) \in \bfP(V_k^*) \times \cdot\cdot \times \bfP(V_k^*)|H_1 \cap \cdot\cdot \cap H_{m+1} \cap M_k \ne \emptyset\}$$
becomes a divisor in $\bfP(V_k^*) \times \cdot\cdot \times \bfP(V_k^*)$,
and this divisor is defined by a polynomial
$${\hat M}_k \in (\mathrm{Sym}^{d_k}(V_k))^{\otimes (m+1)},$$
called the {\bf Chow form}. Consider the $SL(V_k)$-action on $(\mathrm{Sym}^{d_k}(V_k))^{\otimes (m+1)}$.
Stabilizer of ${\hat M}_k$ under $SL(V_k)$-action is $\mathrm{Aut}(M,L)$.
In Theorem \ref{Donaldson} by Donaldson, ``$\mathrm{Aut}(M,L)$ is discrete'' means ``the stabilizer is finite''.

\begin{definition}\label{Chow} Let $(M,L)$ be a polarized manifold.
\begin{enumerate}
\item[1]
$M$ is said to be Chow polystable w.r.t. $L^k$ if the orbit of ${\hat M}_k$ in\\
$(\mathrm{Sym}^{d_k}(V_k))^{\otimes (m+1)}$ under the action of
$\mathrm{SL}(V_k)$ is closed.
\item[2]
$M$ is Chow stable w.r.t $L^k$ if $M$ is polystable and the stabilizer at ${\hat M}_k$
of the action of  $\mathrm{SL}(V_k)$ is finite.
\item[3]
$M$ is Chow semistable w.r.t. $L^k$ if the closure of the orbit of ${\hat M}_k$
in\\
$(\mathrm{Sym}^{d_k}(V_k))^{\otimes (m+1)}$ under the action of
$\mathrm{SL}(V_k)$ does not
contain\\
$\bfo \in (\mathrm{Sym}^{d_k}(V_k))^{\otimes (m+1)}$.
\item[4]
$M$ is asymptotically Chow polystable (resp. stable or semistable) w.r.t. $L$ if there exists a
$k_0 > 0$ such that $M$ is Chow polystable  (resp. stable or semistable) w.r.t. $L^k$ for all $k \ge k_0$.
\end{enumerate}
\end{definition}

In the case when $\mathrm{Aut}(M,L)$ is not discrete Mabuchi tried to extend Theorem \ref{Donaldson} by Donaldson.
He first showed that in this case there is an obstruction to asymptotic Chow semistability:

\begin{theorem}[Mabuchi \cite{mabuchi-a}]\label{mabuchi-a}
Let $(M,L)$ be a polarized manifold.
If $\mathrm{Aut}(M,L)$ is not discrete then
there is an obstruction to asymptotic Chow semistability.
\end{theorem}

This obstruction is expressed in the paper \cite{futaki04-1} as a series of integral
invariants, which are explained later in the next section. Mabuchi then proved the following result.

\begin{theorem}[Mabuchi \cite{mabuchi-c}]\label{mabuchi-c}
Let $(M,L)$ be a polarized manifold, and suppose $\mathrm{Aut}(M,L)$ is not discrete.
If there exists a constant scalar curvature K\"ahler metric in $c_1(L)$ and if the obstruction in Theorem \ref{mabuchi-a}
vanishes then
$(M,L)$ is asymptotically Chow polystable.
\end{theorem}

\section{Obstructions to asymptotic Chow semistability}

The Lie algebra $\frak h_0$ of $\mathrm{Aut}(M,L)$ is expressed in various ways.
Recall that $\frak h(M)$ is the Lie algebra of all holomorphic vector fields on $M$, which is the Lie algebra
of $\mathrm{Aut}(M)$.
First of all it can be expressed as
$$
 \mathfrak h_0 = \{X \in \frak h(M)\ |\ \mathrm{zero}(X) \ne \emptyset\}.
 $$
Secondly it can be expressed also as
$$ \mathfrak h_0  = \{X \in \frak h(M)\ |\ \exists u \in C^{\infty}(M)\otimes\bfC
\ \mathrm{s.t.}\
 X = \mathrm{grad}'u = g^{i\barj}\frac{\partial u}{\partial \barz^j}\frac{\partial}{\partial z^i}\}.$$
 Or we may say that $\mathrm{Aut}(M,L)$ is the linear algebraic part of $\mathrm{Aut}(M)$.
 Mabuchi's obstruction to asymptotic Chow semistability can be re-stated in terms of integral invariants
 $\mathcal F_{\mathrm{Td^i}}$'s, which are explained below,
 as follows.

\begin{theorem}[\cite{futaki04-1}]\label{Futaki04}
Let $(M,L)$ be a polarized manifold with $\dim_{\bfC} M = m$.
\begin{enumerate}
\item[(a)]\ \
The vanishing of Mabuchi's obstruction is equivalent to the vanishing of Lie algebra characters
$\mathcal F_{\mathrm{Td^i}} : \mathfrak h_0 \to \bfC$, for $i = 1, \cdots, m.$
\item[(b)]\ \
$\mathcal F_{\mathrm{Td^1}}$  is an obstruction to the existence of a constant scalar curvature K\"ahler metric in $c_1(L)$,
which is sometimes called the classical Futaki invariant.
\end{enumerate}
\end{theorem}
The Lie algebra characters $\mathcal F_{\mathrm{Td^i}}$ are defined as follows.
For $X \in \mathfrak h_0$ we have
$$ i(X) \omega = - \bar{\partial}\,u_X. $$
Assume the normalization
\begin{equation}\label{normalization}
\int_M u_X\ \omega^m = 0.
\end{equation}
Choose a type $(1,0)$-connection $\nabla$ in $T'M$.
Put
$$ L(X) = \nabla_X - L_X \in \Gamma(\mathrm{End}(T'M))$$
and let
$$ \Theta \in \Gamma(\Omega^{1,1}(M)\otimes\mathrm{End}(T'M))$$
be the (1,1)-part of the curvature form of $\nabla$.
\begin{definition}\label{character}\ \ For $\phi \in I^p(GL(m,\bfC))$, we define
\begin{eqnarray*}
{\mathcal F}_{\phi}(X) &=& (m-p+1) \int_M \phi(\Theta) \wedge u_X\,\omega^{m-p}
\\ & & + \int_M \phi(L(X) + \Theta) \wedge \omega^{m-p+1}.\nonumber
\end{eqnarray*}
\end{definition}
Notice that ${\mathcal F}_{\phi}(X)$ is linear in $X$.
One can show that
${\mathcal F}_{\phi}$ is independent of choices of $\omega$ and $\nabla$.
from which it follows that  ${\mathcal F}_{\phi}$ is invariant under the adjoint action of $\mathrm{Aut}(M)$.
In particular ${\mathcal F}_{\phi}$ is a Lie algebra character.

\begin{proof}[Outline of the proof of Theorem \ref{Futaki04}]\ \ To show (a),
suppose we have a $\bfC^\ast$-action on $M$.
Asymptotic Chow semistablility implies that there is a lift of the $\bfC^\ast$-action to $L$
such that it induces $SL(H^0(L^k))$-action for all $k$.
So, the weight $w_k$ of the action on $H^0(L^k)$ is zero for all $k$.
But $w_k$ can be expressed using the equivariant index formula.
The coefficient of $k^j$ is ${\mathcal F}_{\mathrm{Td^j}}(X)$
where $X$ is the infinitesimal generator of the $\bfC^{\ast}$-action.

To show (b), recall that the first Todd class $\mathrm{Td^1}$ is equal to $\frac12 c_1$. Thus it corresponds to
one half of the trace. Hence the second term of ${\mathcal F}_{\mathrm{Td^1}}(X)$ in Definition \ref{character}
is one half of the integral of the divergence of $X$, which of course vanishes by the divergence theorem. Hence we have
$$ {\mathcal F}_{\mathrm{Td^1}}(X) = \frac m2 \int_M u_X c_1 \wedge \omega^{m-1} $$
where $c_1$ denotes the first Chern form, or the Ricci form. Since $m c_1 \wedge \omega^{m-1} = S \omega^m$ where $S$ is the
scalar curvature, the last integral
becomes zero if $S$ is constant because of the normalization (\ref{normalization}).
This completes the outline of the proof of Theorem \ref{Futaki04}. See \cite{futaki04-1} or \cite{FOS08} for
the detail of the proof.
\end{proof}

Now we have natural questions:

Question (a)\ \ In Theorem \ref{mabuchi-c}, can't we omit the assumption of the vanishing of the obstruction ?
That is to say, if there exists a constant scalar curvature K\"ahler metric in $c_1(L)$ then doesn't the
obstruction necessarily vanish ?

Question (b)\ \ In Theorem\ref{Futaki04}, if $\mathcal F_{\mathrm{Td^1}} = 0$
then $\mathcal F_{\mathrm{Td^2}}=
\cdots = \mathcal F_{\mathrm{Td^m}} = 0$ ?

\bigskip

In \cite{FOS08} we studied the characters $\mathcal F_{\mathrm{Td^i}}$'s in terms of Hilbert series for toric Fano manifolds.
We showed that the linear span of
$\mathcal F_{\mathrm{Td^1}}, \cdots, \mathcal F_{\mathrm{Td^m}}$ coincides with the linear span of the characters obtained as derivatives of
the Hilbert series. Note that the derivatives of the Hilbert series are computed by inputing toric data into a computer.
We saw that, up to dimension three among toric Fano manifolds, there are no counterexamples to Question (b).
But later a seven dimensional
example of Nill and Paffenholz \cite{NillPaffen} appeared, and
Ono, Sano and Yotsutani \cite{OSY09} checked that this seven dimensional example shows that the answers to Questions (a) and (b) are No.
Now we turn to the Hilbert series.

\section{Hilbert series.}

Let
$M$ be a toric Fano manifold of $\dim M = m$. We take
$L = K_M^{-1}$. Then $L$ is a very ample line bundle.
Since $M$ is toric, the real $m$-dimensional torus
$T^m$ acts on $M$ effectively. Since we have a natural $S^1$-action on $K_M^{-1}$, the real $(m+1)$-dimensional torus
$T^{m+1}$ acts on $K^{-1}_M$ effectively so that $K_M^{-1}$ is also toric.

For $g \in T^{m+1}$, we put
$$ L(g) := \sum_{k=0}^\infty \mathrm{Tr}(g|_{H^0(M,K_M^{-k})}).$$
Because of Kodaira vanishing theorem we may regard $L(g)$ as a formal sum of the Lefschetz numbers.
We may analytically continue $L(g)$ to the algebraic torus $T_{\bfC}^{m+1}$, and write it as $L(\bfx)$
for an element $\bfx \in T_{\bfC}^{m+1}$.

Let $\{v_j \in \bfZ^m\}_j$ be the generators of the fan of $M$. Then the moment polytope of $M$ can be expressed as
$$P^{\ast} := \{w \in \bfR^m | v_j\cdot w \ge -1, \forall j\}.$$
 Let
 $$C^{\ast} \subset \bfR^{m+1} (= \mathrm{Lie}(T^{m+1}))^{\ast}$$
 be the cone over $P^{\ast}$.
The integral points in $C^{\ast}$ corresponds bijectively to
the set of all bases of $H^0(M,K_M^{-k})$ for all $k$.

For $\bfx \in T_{\bfC}^{m+1}$ and $\bfa = (w,k) \in \bfZ^{m+1} \cap C^{\ast}$, we put
$$ \bfx^{\bfa} = x_1^{a_1} \cdots x_{m+1}^{a_{m+1}}.$$

\begin{definition}\ \ The Hilbert series $\mathcal C(\bfx, C^{\ast})$ is defined by
$$\mathcal C(\bfx, C^{\ast}) := \sum_{\bfa \in C^{\ast}\cap \bfZ^{m+1}} \bfx^{\bfa}.$$
\end{definition}

The following fact is nontrivial, but is well-known in combinatorics.

\begin{fact}\label{rational}\ \ $\mathcal C(\bfx, C^{\ast})$ is a rational function of $\bfx$.
\end{fact}

It is easy to show the following lemma.

\begin{lemma}\ \ $\mathcal C(\bfx, C^{\ast}) = L(\bfx)$.
\end{lemma}

For $\bfb \in \bfR^{m+1} \cong \mathfrak g = \mathrm{Lie}(T^{m+1})$, put
$$e^{-t\bfb} := (e^{-b_1t},\cdots, e^{-b_{m+1} t}).$$
Then we have
$$\mathcal C(e^{-t\bfb}, C^{\ast}) = \sum_{\bfa \in C^{\ast}\cap \bfZ^{m+1}} e^{-t\bfa\cdot\bfb}.$$
This is a rational function in $t$ by Fact (\ref{rational}).
Let $P$ be the dual polytope of $P^{\ast}$, and put
$$C_{R} := \{(b_1, \cdots, b_m, m+1) | (b_1, \cdots, b_m) \in (m+1)P\} \subset \mathfrak g.$$
An intrinsic meaning of $C_{R}$ can be explained as follows. The unit circle bundle associated with $K_M$ is
considered as a Sasaki manifold with the regular Reeb vector field. But the Reeb vector field can be deformed in
$\mathfrak g$. The subset $C_{R}$ consists of those which are critical points for
the volume functional when we take the variation of the Reeb vector field to be constant multiple of the
Reeb vector field itself (see \cite{MSY2}).
In other words, $C_{R}$ is a natural
deformation space of the Reeb vector fields of the toric Sasaki manifold.

Put $\bfb = (0,\cdots,0,m+1)$.

\begin{theorem}[\cite{FOS08}]\ \
The coefficients of the Laurant series of the rational function
$\frac d{ds}|_{s=0}\mathcal C(e^{-t(\bfb + s\bfc)}, C^{\ast})$
in $t$ span the linear space spanned by
$\mathcal F_{\mathrm{Td}^1}, \cdots,\mathcal F_{\mathrm{Td}^m}$.
\end{theorem}
\noindent
This theorem is a generalization of a result of Martelli, Sparks and Yau \cite{MSY2}, which says the classical Futaki
invariant is obtained as a derivative of the Hilbert series.
Our computations show that the question is
closely related to a question raised by
Batyrev and Selivanova:
Is a toric Fano manifold with vanishing $f (= \mathcal F_{\mathrm{Td}^1})$ for
the anticanonical class necessarily symmetric?
Recall that a toric Fano manifold $M$ is said to be symmetric
if the trivial character is the only fixed point of the action of
the Weyl group on the space of all algebraic characters of
the maximal torus in  $\mathrm{Aut}(M)$. The question of Batyrev and Selivanova is natural because
it is proved by Batyrev and Selivanova \cite{batyrev-selivanova} that if a toric Fano manifold is
symmetric then there exists a K\"ahler-Einstein metric. Later Wang and Zhu \cite{Wang-Zhu} proved
that a toric Fano manifold admits a K\"ahler-Einstein metric if and only if $f (= \mathcal F_{\mathrm{Td}^1})$
vanishes.

Nill and Paffenholz \cite{NillPaffen} gave a counterexample to the question of Batyrev-Selyvanova.
Namely they gave an example of a non-symmetric seven dimensional toric K\"ahler-Einstein Fano manifold on which we have $\mathcal F_{\mathrm{Td}^1}=0$.
Ono, Sano and Yotsutani showed that, in this example, other $\mathcal F_{\mathrm{Td}^i}$'s are non-zero and all proportional.

\section{Higher integral invariants and higher CM lines}

The invariant $\mathcal F_{\mathrm{Td}^1}$ is considered as the Mumford weight of the CM line $\lambda_{CM}$ on
the Hilbert scheme $\mathcal H$ of subschemes of $\bfP^N$
with Hilbert polynomial $\chi$
as shown by Paul and Tian \cite{paultianCM1}, \cite{paultianCM2}. Recently Della Vedova and Zuddas showed
that the same is true for higher $\mathcal F_{\mathrm{Td}^i}$'s. This section is based on their paper \cite{DVZ10}.

Let $(M,L)$ be an $m$-dimensional polarized variety or scheme. For a one parameter
subgroup $\rho : \bfC^{\ast} \to \mathrm{Aut}(M,L)$ with a lifting
to an action $\tilde\rho : \bfC^{\ast} \to \mathrm{Aut}(L)$ on $L$ we denote by
$w(M,L)$ the weight of the induced action on the determinant line $\otimes_{i=0}^m (\det H^i(M,L))^{(-1)^i}$, and by $\chi(M,L)$
the Euler-Poincare characteristic $\sum_{i=0}^m (-1)^i \dim H^i(M,L)$.
Of course if we replace $L$ by its sufficiently high power we may assume
$H^i(M,L) = 0$ for $i > 0$. It is known by the general theory that we have polynomial expansions
\begin{equation}\label{chi}
\chi(M,L^k) = a_0(M,L)k^m + a_1(M,L)k^{m-1} + \cdots + a_m(M,L),
\end{equation}
\begin{equation}\label{weight}
w(M,L^k) = b_0(M,L)k^{m+1} + b_1(M,L)k^m + \cdots + b_{m+1}(M,L).
\end{equation}

We define the Chow weight $ \mathrm{Chow}(M,L^k)$ of $(M,L^k)$ by
$$ \mathrm{Chow}(M,L^k) = \frac{w(M,L^k)}{k\chi(M,L^k)} - \frac{b_0(M,L)}{a_0(M,L)} .$$
One easily gets
\begin{eqnarray*}
\mathrm{Chow}(M,L^k) &=& \frac{b_{m+1}(M,L)}{k\chi(M,L^k)} \\
&& + \frac{a_0(M,L)}{k\chi(M,L^k)} \sum_{\ell=1}^m
\frac{a_0(M,L)b_\ell(M,L)-b_0(M,L)a_\ell(M,L)}{a_0(M,L)^2} k^{m+1- \ell}.
\end{eqnarray*}
The first term $b_{m+1}$ is known to vanish in the smooth case, see \cite{futaki88}. We then define $F_\ell(M,L)$ by
$$ F_\ell(M,L) = \frac {a_0(M,L)b_\ell(M,L) - b_0(M,L)a_\ell(M,L)}{a_0(M,L)^2}. $$
If $M$ is smooth, $\chi(M,L)$ is expressed using Todd classes and $c_1(L)$ by Riemann-Roch theorem and $w(M,L)$ is expressed using
Todd classes, $c_1(L)$, connections in the tangent bundle of $M$ and $L$ with the infinitesimal action of $X$.
The connection term in $L$ makes its appearance as the
Hamiltonian function $u_X$ in Definition \ref{character} of $\mathcal F_\phi(X)$. Hence the terms $a_i(M,L)$ and $b_j(M,L)$ are written
in terms those classes and connections. Della Vedova and Zuddas show that $F_\ell(M,L) $ is independent of the choice of a lifting
$\tilde\rho : \bfC^{\ast} \to \mathrm{Aut}(L)$ of $\rho$ and that
\begin{equation}\label{DVZformula}
F_\ell(M,L) = \frac{1}{vol(M, L)}\mathcal F_{\mathrm{Td}^\ell}(X)
\end{equation}
when $M$ is smooth and $X$ is the infinitesimal generator of the action $\rho : \bfC^{\ast} \to \mathrm{Aut}(M,L)$.
We give here the case when $\ell = 1$. Refer to \cite{DVZ10} for general $\ell$.

\begin{lemma}[\cite{donaldson02}]\ \ If  $M$  is a nonsingular projective variety then
$$ F_1(M,L) = \frac{1}{vol(M, L)}\mathcal F_{\mathrm{Td}^1}(X) $$
where $X$ is the infinitesimal generator of the $\bfC^{\ast}$-action.
\end{lemma}
\begin{proof}\ \ Let us denote by $m$ the complex dimension of $M$.
Expand $h^0(L^k)$ and $w(k)$ as
$$ h^0(L^k) = a_0k^m + a_1k^{m-1}+ \cdots,$$
$$ w(k) = b_0k^{m+1} + b_1k^m + \cdots.$$
Then by the Riemann-Roch and the equivariant Riemann-Roch formulae
$$ a_0 = \frac 1{m!}\int_M c_1(L)^m = vol(M), $$
$$ a_1 = \frac 1{2(m-1)!} \int_M \rho \wedge c_1(L)^{m-1} = \frac 1{2m!} \int_M \sigma \omega^m, $$
$$ b_0 = \frac 1{(m+1)!}\int_M (m+1) u_X \omega^m, $$
$$ b_1 = \frac 1{m!} \int_M m u_X \omega^{m-1} \wedge \frac 12 c_1(M) + \frac 1{m!} \int_M
\operatorname{div}X\ \omega^m .$$
The last term of the previous integral is zero because of the divergence formula. Thus
$$ \frac {w(k)}{kh^0(k)} = \frac{b_0}{a_0}(1 + (\frac{b_1}{b_0} - \frac{a_1}{a_0})k^{-1} + \cdots ) $$
from which we have
\begin{eqnarray*}
 F_1(M,L) &=&
 \frac{b_0}{a_0}(\frac{b_1}{b_0} - \frac{a_1}{a_0}) = \frac 1{a_0^2}(a_0b_1 - a_1b_0)\\
 &=& \frac 1{2vol(M, L)}\int_M u_X(\sigma - \frac 1{vol(M, L)}\int_M \sigma \frac{\omega^n}{n!})\frac{\omega^n}{n!}\\
 &=&  \frac{1}{vol(M, L)}\mathcal F_{\mathrm{Td}^1}(X)
\end{eqnarray*}
\end{proof}

\begin{definition}\label{DF} Let $(M,L)$ be a polarized scheme.
We call $F_1(M,L)$ the DF-invariant of $(M,L)$.
\end{definition}

The DF-invariant $F_1(M,L)$ is used as a numerical invariant to define K-stability, see next section for the detail. The idea is the similar to the following
Hilbert-Mumford criterion for Chow stability.


Let $f : \mathcal U \to \mathcal H$ be the universal flat family over the Hilbert scheme $\mathcal H$ of subschemes of $\bfP^N$
with Hilbert polynomial $\chi$, and $\iota : \mathcal U \to \mathcal H \times \bfC\bfP^N$ be
the natural embedding. Then we have $f = \mathrm{pr}_{\mathcal H}\circ \iota$. Let $\mathcal L = \iota^\ast\circ \mathrm{pr}_{\mathcal H}^{\ast}
\mathcal O(1)$ be the relatively ample line bundle over $\mathcal U$. For $k$ sufficiently large we have
$\mathrm{rank} f_\ast (L^k) = \dim H^0(\mathcal U_x, \mathcal L_x^k)$ and
$\det f_\ast(L^k) = \det H^0(\mathcal U_x, \mathcal L_x^k)$
for all $x \in \mathcal H$. Hence we have
\begin{equation}\label{rank}
\mathrm{rank} f_\ast (L^k) = a_0 k^n + a_1 k^{n-1} + \cdots + a_n.
\end{equation}
Considering the determinant we see from \cite{KnudsenMumford} that there are $\bfQ$-line bundles $\mu_0, \cdots, \mu_{m+1}$ such that
\begin{equation}\label{det}
\mathrm{det} f_\ast (L^k) = \mu_0^{k^{m+1}}\otimes \mu_1^{k^m} \otimes \cdots \otimes \mu_{m+1}.
\end{equation}

By definition {\it Chow-line} is the $\bfQ$-line bundle $\lambda_{Chow} (\mathcal H, \mathcal L)$ over $\mathcal H$
\begin{equation}\label{Chow-line}
\lambda_{Chow} (\mathcal H, \mathcal L) = \det f_\ast(\mathcal L) ^{\frac 1{k\mathrm{rank}f_\ast(\mathcal L)} }\otimes \mu_0 ^{- \frac 1 {a_0}}.
\end{equation}
It is easy to see that $ \mathrm{Chow}(M,L)$ is the Mumford weight of the Chow-line $\lambda_{Chow} (\mathcal H, \mathcal L)$.
By (\ref{rank}) and (\ref{det}) one can show
\begin{equation}\label{Chow-line2}
\lambda_{Chow} (\mathcal H, \mathcal L) = \mu_{m+1}^{\frac 1{k\chi(k)}} \otimes \left(\bigotimes_{\ell = 1}^m \left(\mu_\ell^{\frac1{a_0}} \otimes
\mu_0^{-\frac{a_\ell}{a_0^2}}\right)^{\frac{a_0k^{m+1-\ell}}{\chi(k)}}\right).
\end{equation}
We define the $\ell$-th CM-line $ \lambda_{\mathrm{CM}, \ell}(\mathcal H, \mathcal L) $ on the Hilbert scheme $\mathcal H$ by
$$ \lambda_{\mathrm{CM}, \ell}(\mathcal H, \mathcal L) = \mu_\ell^{\frac 1{a_0}} \otimes \mu_0^{- \frac {a_\ell}{a_0^2}}.$$
It is also easy to see that $F_\ell(M,L)$ is the weight of the  $\ell$-th CM-line $ \lambda_{\mathrm{CM}, \ell}(\mathcal H, \mathcal L) $.

Della Vedova and Zuddas then compute $ \mathrm{Chow}(M,L)$ and $F_\ell(M,L)$ for projective bundles over curves and for polarized manifolds
blown-up at finite points.

Let $\Sigma$ be a genus $g$ smooth curve and $E$ a rank $n \ge 2$ vector bundle over $\Sigma$. Let $M = \bfP(E)$ be the projective bundle
associated to $E$ and  denote by $\pi : M \to \Sigma$ the projection. A line bundle $L$ on $M$ is the form $L = \mathcal O_{\bfP(E)}(r) \otimes \pi^\ast B$
where $B$ is a line bundle over $\Sigma$. We assume that $L$ is ample. We also assume that $E$ is decomposed as $E = E_1 \oplus \cdots \oplus E_s$
into indecomposable components $E_i$, and that we are given a $\bfC^\ast$ action on $E$ written in terms of this decomposition
$$ t\cdot (e_1, \cdots, e_s) = (t^{\lambda_1}e_1, \cdots, t^{\lambda_s} e_s). $$
In this situation $\mathrm{Chow}(M,L^k)$ is given by
\begin{eqnarray}\label{Chow-bundle}
\mathrm{Chow}(M,L^k) &=& \frac{\binom{m-1+kr}{m}}{m+1} \frac{\chi(\Sigma, \det(E\otimes B^{-\frac1r})}{\mu(E\otimes B^{-\frac1r})\chi (\Sigma, S^{kr}(E^\ast\otimes B^{\frac1r}))} \\
& & \qquad \cdot\sum_{j=1}^s \lambda_j \mathrm{rank}(E_j)(\mu(E_j) - \mu(E)) \nonumber
\end{eqnarray}
where $\mu(F) = \mathrm{deg}(F) / \mathrm{rank}(F)$ is the slope of the bundle $F$. On the other hand $F_\ell(M,L)$ is computed for some positive rational number
depending only on $m$ as
\begin{equation}\label{F-ell-bundle}
F_\ell(M,L^k) = - C_\ell\ \frac{\chi(\Sigma, \det(E\otimes B^{-\frac1r}))}{\mu(E\otimes B^{-\frac1r})^2} \sum_{j=1}^s \lambda_j \mathrm{rank}(E_j)(\mu(E_j) - \mu(E)).
\end{equation}
By (\ref{Chow-bundle}) and (\ref{F-ell-bundle}) we see that $F_\ell(M,L^k) $ are proportional for all $\ell$, that they vanish if and only if $\mu(E_j) = \mu(E)$ for
all $j = 1, \cdots, s$, and that $\mathrm{Chow}(M,L^k) = 0$ if and only if $F_\ell(M,L^k) = 0$ for some (and hence any) $\ell$.

The slope stability of $E$ is related to the existence of cscK metric as in the following theorem.
\begin{theorem}[\cite{ACGT0905}]\label{ACGTtheorem}
A projective bundle $\bfP(E)$ over a smooth curve of genus $g \ge 2$ admits a K\"ahler metric of constant scalar curvature in some (and hence any)
K\"ahler class if and only if $E$ is slope polystable.
\end{theorem}


We will not reproduce the formulas of $ \mathrm{Chow}(M,L)$ and $F_\ell(M,L)$ for polarized manifolds obtained by
blowing-up at finite points, but the consequences of the formulas are summarized as follows.
By a result of LeBrun and Simanca \cite{lebrunsimanca93} the cone $\mathcal E$ of extremal K\"ahler classes is open in the K\"ahler cone,
and the locus where the Futaki invariant $F_1$ vanishes is the set $\mathcal C$ of all cscK classes. By the results of Arezzo and Pacard \cite{arezzopacard06},  \cite{arezzopacard09} there is a non-empty open set of cscK classes under mild conditions. Under such conditions we may be able to show
that the locus  $\mathcal Z$ where $F_2 = \cdots = F_m = 0$ is a Zariski closed subset in $\mathcal C$. Then a rational point
in $\mathcal C \backslash \mathcal Z$ will be a cscK but asymptotically unstable polarization. This idea works for the blow-up of $\bfC\bfP^2$ at
four points with all but one aligned. See \cite{DVZ10} for the detail.

\section{Toric case}

In this section we compare H.Ono's paper \cite{Ono10-1} with the work of Della Vedova and Zuddas \cite{DVZ10}.
Let $\Delta \subset \bfR^m$ be an $m$-dimensional integral Delzant polytope. Namely, \\
(i) $\Delta$ has integral vertices ${\bf w}_1, \cdots, {\bf w}_d$, \\
(ii) $m$ edges of $\Delta$ emanate from each vertex ${\bf w}_i$, and \\
(iii) primitive vectors along those edges generate the lattice $\bfZ^m \subset \bfR^m$. \\
To a Delzant polytope there correspond a nonsingular toric variety and an
ample line bundle $L$.
The Ehrhart polynomial of $\Delta$
\begin{equation}\label{Ehrhart}
E_P(k) = \mathrm{Vol}(\Delta)k^n + \sum_{j=0}^{m-1} E_{P,j}k^j
\end{equation}
has the property that
$$E_P(i) = \sharp (iP\cap\bfZ^m).$$
It is also known that there exists an $\bfR^m$-valued polynomial
\begin{equation}\label{s(k)}
{\bf s}_\Delta(k) = k^{n+1} \int_\Delta \bfx\,dv + \sum_{j=1}^m k^j\,\bfs_{\Delta, j}
\end{equation}
such that
\begin{equation}
\bfs_\Delta (i) = \sum_{\bfa \in i\Delta \cap \bfZ^m} \bfa.
\end{equation}
Then Ono \cite{Ono10-1} proves that if, for each $i$, $(M_\Delta, L_{\Delta}^i)$ is (not necessarily asymptotically) Chow semistable, we have
\begin{equation}\label{Onoformula}
{\bf s}_\Delta (i) = \frac{E_\Delta(i)}{\mathrm{Vol}(i\Delta)}\int_{i\Delta} \bfx\,dv.
\end{equation}
Hence if $(M_\Delta, L_{\Delta})$ is asymptotically Chow semistable, we have the equality
\begin{equation}\label{Onoformula2}
\mathrm{Vol}(\Delta){\bf s}_\Delta (k) - kE_\Delta(k)\int_{\Delta} \bfx\,dv =
\sum_{j=0}^m k^j \left(\mathrm{Vol}(\Delta)\bfs_{\Delta,j} - E_{\Delta,j-1}\int_\Delta \bfx \,dv\right)
= 0
\end{equation}
as a polynomial in $k$. But the Ehrhart polynomial is equal to the Hilbert polynomial $\chi(M_\Delta, L_{\Delta}^k)$.
Moreover, ${\bf s}_\Delta (k)$ can be regarded as a character
of the torus and gives the weight $w(M_\Delta, L_\Delta^k)$ on
$L_{\Delta}^k$ when restricted to a one parameter subgroup. Therefore, as a character,
$$\mathrm{Vol}(k\Delta){\bf s}_\Delta (k) - kE_\Delta(k)\int_{\Delta} \bfx\,dv$$
is equal to
$$\mathrm{Vol}(\Delta)w(M_\Delta, L_\Delta^k) - k\chi(M_\Delta, L_{\Delta}^k)\int_{M_\Delta} u_X \omega^m$$
when restricted to the one parameter group generated by an infinitesimal generator $X$.
Put
\begin{equation}\label{Onoformula3}
\mathcal{F}_{\Delta,j} := \mathrm{Vol}(\Delta) \bfs_{\Delta,j} - E_{\Delta, j-1}\int_\Delta\,dv \in \bfR^m.
\end{equation}
By (\ref{Onoformula2}), $\mathcal{F}_{\Delta,j}$ vanishes if $(M_\Delta, L_{\Delta}^i)$ is Chow semistable.
But (\ref{DVZformula}) shows
\begin{equation}\label{Onoconjecture}
\mathrm{Lin}_\bfC\{\mathcal F_{\Delta,j},\ j=1, \cdots, m \} = \mathrm{Lin}_\bfC\{\mathcal F_{\mathrm{Td}^{(p)}}|_{\bfC^m},\ p=1,\cdots, m\}
\end{equation}
where $\mathrm{Lin}_\bfC$ stands for the linear hull in $\bfC^m$. This gives a proof to Conjecture 1.6 in \cite{Ono10-1}.

In \cite{Ono10-2}, Ono further gives a necessary and sufficient condition for Chow semistability condition for $(M_\Delta, L_\Delta^i)$
in terms of toric data. Shelukhin \cite{Shelukhin09} also expresses $\mathcal F_1(M,-K_M)$ for a toric Fano manifold $M$ in terms of toric data
of $M$.

\section{K stability}

The notion of K-stability was first introduced by Tian in \cite{tian97}
for Fano manifolds
and proved that if a Fano manifold carries a K\"ahler-Einstein metric then
$M$ is weakly K-stable. Tian's K-stability considers the degenerations of
$M$ to
normal varieties and uses a generalized version of the invariant $\mathcal F_1$
which were defined for normal varieties.
Donaldson re-defined in \cite{donaldson02} the invariant $\mathcal F_1$
for general
polarized varieties (or even projective schemes) as introduced in the previous section, and also re-defined
the notion of K-stability for a polarized manifold $(M, L)$.

For a polarized variety $(M, L)$, a test configuration of
degree $r$ consists of the following.\\
(a)\ \ A flat family of schemes $\pi : {\mathcal M} \to \bfC$:\\
(b)\ \ $\bfC^*$-action on ${\mathcal M}$ such that $\pi : {\mathcal M} \to \bfC$ is $\bfC^\ast$-equivariant with respect to the usual $\bfC^*$-action on $\bfC$:\\
(c)\ \ $\bfC^*$-equivariant relatively ample line bundle $\mathcal{L} \to {\mathcal M}$ such that
for $t \ne 0$ one has $M_t = \pi^{-1}(t) \cong M$ and
$(M_t, {\mathcal L}|_{M_t}) \cong (M, L^r)$.

$\bfC^*$-action on $(\mathcal M, \mathcal L)$ induces a $\bfC^*$-action on the central fiber
$L_0 \to M_0 = \pi^{-1}(0)$. Moreover if
$(M,L)$ admits a $\bfC^*$-action, then one obtains a test configuration
by taking the direct product $L^r \times \bfC \to M \times \bfC$. This is called a product configuration.
A product configuration endowed with the trivial $\bfC^\ast$ action is called the trivial configuration.

\begin{definition}Let $(M,L)$ be a polarized variety, and $(\mathcal M, \mathcal L)$ a test configuration of
$(M,L)$. We define DF-invariant $DF(\mathcal M, \mathcal L)$ to be the DF-invariant $F_1(M_0,L_0)$ of the central fiber $(M_0, L_0)$.
\end{definition}

\begin{definition}A polarized variety $(M,L)$ is said to be  K-polystable (resp. stable) if
the DF-invariant $DF(\mathcal M, \mathcal L)$ is negative or equal to zero
for all test configurations $(\mathcal M, \mathcal L)$, and the equality occurs only
if the test configuration is product (resp. trivial).
\end{definition}

\noindent
{\bf Conjecture}(\cite{donaldson02}) : Let $(M,L)$ be a nonsingular polarized variety. Then a K\"ahler metric of constant scalar curvature will exist in the
K\"ahler class $c_1(L)$  if and only if $(M,L)$ is K-polystable.

\bigskip

Let us recall the following general terminology. Let $V$ be a vector space over $\bfC$ and $\rho$ a one parameter
subgroup of $SL(V)$. Let $[v] \in \bfP(V)$ and $\lambda \in \bfC^{\ast}$.
Suppose $[\rho(\lambda)v] \to [v_0] \in \bfP(V)$ as $ \lambda \to 0$. Then
we have an endomorphism $\rho(\lambda):\bfC v_0 \to \bfC v_0$.
The weight of this endomorphism is called Mumford weight of $(v,\rho)$ and
is denoted by $\mu(v,\rho)$. We say that $[v] \in \bfP(V)$ is semistable (resp.
stable) with respect to $\rho$ iff $\mu(v,\rho) \le 0$
(resp. $\mu(v,\rho) < 0$). We also say that $[v] \in \bfP(V)$ is polystable
iff  $\mu(v,\rho) < 0$ or $\rho(\bfC^{\ast})$ is contained in $Stab(v)$.
The Hilbert-Mumford criterion says that $[v] \in \bfP(V)$ is semistable (resp.
polystable) with respect to a subgroup $G$ of $SL(V)$ iff
$[v] \in \bfP(V)$ is semistable (resp.
polystable) with respect to arbitrary one parameter subgroup of $G$.

Let us define Hilbert stability of a polarized variety $(M,L)$. Suppose $L^r$
is a very ample line bundle with $h^i(L^r) = 0$ for $i > 0$.
Then $\chi(r) := h^0(L^r)$ can be computed by Riemann-Roch theorem. If we fix an
isomorphism $H^0(L^r) \cong \bfC^{\chi(r)}$ this gives an embedding
$\Phi_{|L^r|} : M \to \bfP^{\chi(r) -1}$. A different choice of the isomorphism
gives a transformation by an element of $SL(\chi(r))$. When $k$ is sufficiently
large we have an exact sequence
$$ 0 \to I_k \to S^kH^0_M(L^r) \to H^0_M(L^{kr}) \to 0, $$
where $I_k$ denotes the set of all polynomials of degree $k$ vanishing along the image of $M$.
The $k$-th Hilbert point of $(M,L^r)$ is the point in the Grassmannian
$$ x_{k,r} \in G = G(S^k\bfC^{\chi(r)\ast};\chi(rk))$$
determined by the identification $H^0_M(L^r) \cong \bfC^{\chi(r)}$.

We say that $(M,L)$ is Hilbert (semi)stable with respect to $r$ iff the image
of $x_{r,k} \in G$ of the Pl\"ucker embedding $G \to \bfP^{\binom{\chi(r)+k-1}{\chi(rk)}}$ is (semi)stable for all large $k$.

\begin{fact}[c.f. \cite{mumford}, Proposition 2.1]\ \ Let $L$ be a very ample
line bundle with $h^i(L) = 0$ for $i > 0$, and $\rho$ a one parameter
subgroup of $SL(h^0(L))$. Let ${\widetilde w}$ be the Mumford weight of the
Hilbert point $x_k \in  G(S^k\bfC^{h^0(L)\ast};\chi(k))$ with respect to $\rho$,
and $e$ be the Mumford weight of the Chow point of $(M,L)$ with respect to
$\rho$. Then we have
$$ \tildew (k) = Cek^{m+1} + O(k^m)$$
with positive constant $C$.
\end{fact}

This says if $e < 0$ then $\tildew(k) <0$ for large $k$, namely Chow stability
implies Hilbert stability. If $\tildew(k) \le 0$ for all $k$, then $e \le 0$,
namely Hilbert semistable implies Chow semistable.

Now let $\tildew(r,k)$ be the Mumford weight of $x_{r,k}$. We wish to
express this in terms of $w(r)$ which was the weight for $H^0(L^r)$ of
the one parameter
group $\rho$ in $SL(h^0(L))$. As $\rho$ lies in $SL(h^0(L))$ we have to
renormalize the one parameter
group so that in lies in $SL(h^0(L^r))$. After this renormalization we find
by putting $s = rk$
\begin{eqnarray*}
\tildew(r,k) &=& - w(s) + \frac{w(r)}{r\chi(r)}s\chi(s)\\
&=& s\chi(s)( \frac{w(r)}{r\chi(r)} -  \frac{w(s)}{s\chi(s)} )\\
&=& s\chi(s)(F_1(r^{-1} - s^{-1}) + O(r^{-2} - s^{-2})).
\end{eqnarray*}

\begin{theorem}[\cite{ross03}, \cite{rossthomas06}]\ \ If we put
$ \tildew(r,k) = \frac 1{r\chi(r)} \sum_{i,j = 0}^{m+1} a_{i,j}r^{i+j}k^j $
then
\begin{enumerate}
\item $a_{m+1, m+1} = 0$:
\item The Chow weight $e_r := {\mathrm Chow}(M,L^r)$ of $(M,L^r)$ is given by
$$e_r = \frac{Cr^m}{\chi(r)} \sum_{i=0}^m a_{i,m+1}r^i$$
with a positive constant
$C$:
\item $a_{m,m+1}$ and $F_1(M,L)$ have the same sign.
\end{enumerate}
\end{theorem}

This result says that if $e_r \le 0$ for all large $r$ then $F_1(M,L) \le 0$,
namely that asymptotic Chow semistability implies K-semistability.

Now we turn to the computation of $F_1(M,L)$.
The following result of Wang gives a way of computing $F_1(M,L)$. Note that the sign convention for the DF-invariant is opposite
in \cite{Xiaowei08},  \cite{Odaka09} and  \cite{Odaka10}.
\begin{theorem}[Wang \cite{Xiaowei08}]\label{Wang}
For any test configuration $(\mathcal M, \mathcal L)$ of a polarized variety $(M,L)$ we consider its natural compactification
$(\overline{\mathcal M}, \overline{\mathcal L})$. Then $F_1(\mathcal M, \mathcal L)$ is computed by
\begin{equation*}
DF(\mathcal M, \mathcal L) = \frac {-1}{2(m!)((m+1)!)} ( - m(L^{m-1}. K_M)(\overline{\mathcal L}^{m+1}
 + (m+1)(L^m)(\overline{\mathcal L}^m.K_{\overline{\mathcal M}/\bfP^1}))
\end{equation*}
where $K_{\overline{\mathcal M}/\bfP^1} = K_{\overline{\mathcal M}} - f^\ast K_{\bfP^1}$ with the projection $f : \overline{\mathcal M} \to \bfP^1$.
The notation $(L^m)$ means the intersection number $L\ldots L$ ($m$ times) in $M$, and so on.
\end{theorem}

With different technicalities Odaka extends and applies this result to the semi test configuration $\mathcal B := Bl_{\mathcal J}(M\times \bfC)$
obtained by blowing up the flag ideal $\mathcal J \subset \mathcal O_{\mathcal X \times \bfC}$
of the form
$$ \mathcal J = I_0 + I_1 t + I_2 t^2 + \cdots + I_{N-1} t^{N-1} + (t^N) $$
where $I_0 \subset I_1 \subset \cdots \subset I_{N-1} \subset \mathcal O_M$ is a sequence of coherent ideals of $M$.
Denote this blow-up by $\Pi : \mathcal B \to M \times \bfC$ and by $E$ the exceptional divisor, i.e., $\mathcal O(-E) = \Pi^{-1}\mathcal J$.
We also put $\mathcal L := p_1^\ast L$ where $p_i $ is the projection of $M\times \bfC$ or $M \times \bfP^1$ to the $i$-th factor.
We assume that the restriction of $\mathcal L(-E)$ to $\mathcal B$ is relatively semiample, and hence we have a semi
test configuration $(\mathcal B, \mathcal L(-E)|_{\mathcal B})$.
$\mathcal B$ is compactified to $\overline{\mathcal B}:= Bl_{\mathcal J}(X \times \bfP^1)$.
Then the DF-invariant $DF(\mathcal B, \mathcal L(-E))$ is computed as follows.
\begin{theorem}[Odaka \cite{Odaka09}]\label{Odaka1}
\begin{eqnarray*}
& &DF(\mathcal B, \mathcal L(-E)) = \frac {-1}{2(m!)((m+1)!)} ( - m(L^{m-1}. K_M)(\mathcal L(-E))^{m+1} \\
 &&\qquad \qquad + (m+1)(L^m)(\mathcal L(-E)^m.p_1^\ast K_M) + (m+1)(L^m)(\mathcal L(-E)^m.K_{\mathcal B/M\times \bfC}))
\end{eqnarray*}
where the intersection numbers are taken on $M$ or $\overline{\mathcal B}$.
\end{theorem}

The next theorem shows that this computation is sufficient to check K-(semi)stability.
\begin{theorem}[Odaka \cite{Odaka09}]\label{Odaka2}
The negativity (resp. nonpositivity) of all the DF-invariance of the semi test configurations of the above blow-up type
$(\mathcal B, \mathcal L(-E))$ with $\mathcal B$ Gorenstein in codimension 1 is equivalent to K-stability (resp. K-semistability)
of $(M,L)$.
\end{theorem}

In \cite{Odaka10}, Odaka proves Theorem \ref{Odaka} using Theorem \ref{Odaka1} and \ref{Odaka2}. He also proves in \cite{Odaka10}\\
$\bullet$\ A semi-log-canonical canonically polarized variety $(X,\mathcal O_X(mK_X))$ with $m \in \bfZ_{>0}$ is K-stable.\\
$\bullet$\ A log-terminal polarized variety $(X,L)$ with numerically trivial canonical divisor $K_X$ is K-stable.\\
These results are expected to be true because of Calabi-Yau theorem \cite{yau78}. In \cite{OdakaSano10}, Odaka and Sano give an algebro-geometric proof
of the fact that if the alpha invariant of a Fano manifold $M$, which is equal to the log canonical threshold, is bigger than
$m/(m+1)$ then $(M,- K_M)$ is K-stable. This is of course another proof of a consequence of a theorem of Tian \cite{tian87}.

\bibliographystyle{amsalpha}

\end{document}